\documentclass{amsart}

\usepackage{amsmath, amsthm}
\usepackage{amssymb}
\usepackage{amsfonts}
\usepackage{latexsym}
\usepackage{mathrsfs}
\usepackage{bm}
\usepackage{graphicx}

\theoremstyle{plain}
\newtheorem{theorem}{Theorem}[section]

\newtheorem{lemma}[theorem]{Lemma}
\newtheorem{definition}[theorem]{Definition}
\newtheorem{proposition}[theorem]{Proposition}

\theoremstyle{remark}
\newtheorem{remark}{Remark}[section]

\DeclareMathOperator{\Teich(S)}{Teich(S)}

\newcommand{\Mod}{\mbox{\rm Mod}}

\begin{document}

\date{}
\title[Curvature operator]
{The Riemannian Sectional Curvature Operator Of The Weil-Petersson Metric  and Its Application}

\author{Yunhui Wu}

\address{Department of Mathematics\\
         Rice University\\
         Houston, Texas, 77005-1892\\}
\email{yw22@rice.edu}

\begin{abstract}
Fix a number $g>1$, let $S$ be a close surface of genus $g$, and $\Teich(S)$ be the Teichm\"uller space of $S$ endowed with the Weil-Petersson metric. In this paper we show that the Riemannian sectional curvature operator of $\Teich(S)$ is non-positive definite. As an application we show that any twist harmonic map from rank-one hyperbolic spaces $H_{Q,m}=Sp(m,1)/Sp(m)\cdot Sp(1)$ or $H_{O,2}=F_{4}^{-20}/SO(9)$ into $\Teich(S)$ is a constant.
\end{abstract}


\maketitle

\section{Introduction}
Let $S$ be a closed surface of genus $g$ where $g>1$, and $T_{g}$ be the Teichm\"uller space of $S$. $T_{g}$ carries various metrics that have respective properties. For example, the Teichm\"uller metric is a complete Finsler metric. The McMullen metric, Ricci metric, and perturbed Ricci metric have bounded geometry \cite{LSY04, LSY05, McM00}. The Weil-Petersson metric is K\"ahler \cite{Ahlfors} and incomplete \cite{Chu,  Wolpert1}. There are also some other metrics on $T_{g}$ like the Bergman metric, Caratheodory metric, K\"ahler-Einstein metric, Kobayashi metric, and so on. In \cite{LSY04,LSY05}, the authors showed that some metrics listed above are comparable. In this paper we focus on the Weil-Petersson case. Throughout this paper, we let $\Teich(S)$ denote $T_{g}$ endowed with the Weil-Petersson metric. The geometry of the Weil-Petersson metric has been well studied in the past decades. One can refer to Wolpert's recent nice book \cite{Wolpert4} for details.

The curvature aspect of $\Teich(S)$ is very interesting, which plays an important role in the geometry of Weil-Petersson metric. This aspect has been studied over the past several decades. Ahlfors in \cite{Ahlfors} showed that the holomorphic sectional curvatures are negative. Tromba \cite{Tromba} and Wolpert \cite{Wol86} independently showed the sectional curvature of $\Teich(S)$ is negative. Moreover, in \cite{Wol86} the author proved the Royden's conjecture, which says that the holomorphic curvatures are bounded above by a negative number that only depends on the topology of the surface, by establishing the curvature formula (see theorem \ref{cfow}). Wolf in \cite{Wolf89} used harmonic tools to give another proof of this curvature formula. After that, people has been applying this formula to study the curvature of $\Teich(S)$ in more detail. For example, in \cite{Schumacher} Schumacher showed that $\Teich(S)$ has strongly negative curvature in the sense of Siu (see \cite{Siu}) that is stronger than negative sectional curvature. Huang in \cite{Huang1} showed there is no negative upper bound for the sectional curvature. In \cite{LSY} Liu-Sun-Yau also used Wolpert's curvature formula to show that $\Teich(S)$ has dual Nakano negative curvature, which says that the complex curvature operator on the dual tangent bundle is positive in some sense. For some other related problems one can refer to \cite{BF06,Huang1,Huang,LSY04,LSY05,Teo,Wolpert3,Wolpert5}.

Let $X \in \Teich(S)$. We can view $X$ as a hyperbolic metric on $S$. One of our purposes in this paper is to study the Riemannian sectional curvature operator of $\Teich(S)$ at $X$. The method in this paper is highly influenced by the methods in \cite{LSY,Schumacher,Wol86}, which essentially applied the curvature formula, the Cauchy-Schwarz inequality and the positivity of the Green function for the operator $(\Delta-2)^{-1}$, where $\Delta$ is the Beltrami-Laplace operator on $X$. What we need more in this paper is the symmetry of the Green function for $(\Delta-2)^{-1}$. 

Before giving any statements let us state some neccessary background. Let $X$ be a point in $\Teich(S)$, and $T_{X}\Teich(S)$ be the tangent space that is identified with the harmonic Beltrami differentials at $X$. Assume that $\{\mu_{i}\}_{i=1}^{3g-3}$ is a basis for $T_{X}\Teich(S)$, and $\frac{\partial}{\partial t_{i}}$ is the vector fields corresponding to $\mu_{i}$. Locally, $t_{i}$ is a holomorphic coordinate around $X$; let $t_{i}=x_{i}+\textbf{i}y_{i}$.  $(x_{1},x_{2},...,x_{3g-3},y_{1},y_{2},...,y_{3g-3})$ gives a real smooth coordinate around $X$. Since $\Teich(S)$ is a Riemannian manifold, it is natural to define the curvature tensor on it, which is denoted by $R(\cdot,\cdot,\cdot,\cdot)$. Let $T\Teich(S)$ be the real tangent bundle of $\Teich(S)$ and $\wedge^{2}T\Teich(S)$ be the wedge product of two copies of $T\Teich(S)$. The \textsl{curvature operator} $Q$ is defined on $\wedge^{2}T\Teich(S)$ by 
$Q(V_{1}\wedge V_{2},V_{3}\wedge V_{4})=R(V_{1},V_{2},V_{3},V_{4})$ and extended linearly, where $V_{i}$ are real vectors. It is easy to see that $Q$ is a bilinear symmetric form (one can see more details in \cite{Jost}).  

Now we can state our first result. 
\begin{theorem}\label{conp}
Let $S=S_{g}$ be a closed surface of genus $g>1$ and $\Teich(S)$ be the Teichm\"uller space of $S$ endowed with the Weil-Petersson metric. And let $\textbf{J}$ be the almost complex structure on $\Teich(S)$ and $Q$ be the curvature operator of $\Teich(S)$. Then, for any $X \in \Teich(S)$, we have\\
(1) $Q$ is non-positive definite, i.e., $Q(A,A)\geq 0$ for all $A \in \wedge^{2}T_{X}\Teich(S)$.\\
(2) $Q(A,A)=0$ if and only if there exists an element $B$ in $\wedge^{2}T_{X}\Teich(S)$ such that
$A=B-\textbf{J}\circ B$.
\end{theorem}
\noindent $\textbf{J}\circ B$ is defined in section \ref{4}. \\

A direct corollary is that the sectional curvature of $\Teich(S)$ is negative \cite{Ahlfors, Tromba, Wol86}. Normally a metric of negative curvature may not  have non-positive definite curvature operator (see \cite{AF}).
\newline

In the second part of this paper we will study harmonic maps from certain rank-one spaces into $\Teich(S)$. For harmonic maps, there are a lot of very beautiful results when the target is either a complete Riemannian manifold with non-positive curvature operator or a complete non-positive curved metric space (see \cite{Cor,DW03,Yamada04}). In particular, if the domain is either the Quaternionic hyperbolic space or the Cayley plane, different rigid results for harmonic maps were established in \cite{GS, JY, MSY}. For harmonic maps into $\Teich(S)$, one can refer to the nice survey \cite{DW09}. In this paper we establish the following rigid result.

\begin{theorem}\label{hmrt}
Let $\Gamma$ be a lattice in a semisimple Lie group $G$ which is either $Sp(m,1)$ or $F_{4}^{-20}$, and $\Mod(S)$ be the mapping class group of $\Teich(S)$. Then, any twist harmonic map $f$ from $G/\Gamma$ into $\Teich(S)$ with respect to each homomorphism $\rho: \Gamma \rightarrow \Mod(S)$ must be a constant.
\end{theorem}
\noindent The twist map $f$ with respect to $\rho$ means that $f(\gamma \circ Y)=\rho (\gamma)\circ f(Y)$ for all $\gamma \in \Gamma$.\newline

\subsection*{Plan of the paper.} In section \ref{np} we provide some necessary background and some basic properties for the operator $D=-2(\Delta-2)^{-1}$. In section \ref{coos} we establish the curvature operator formulas on different subspaces of $\wedge^{2}{T_{X}\Teich(S)}$ and show that the curvature operator is negative definite or non-positive definite on these different subspaces. In section \ref{4} we establish the curvature operator formula for $Q$ on $\wedge^{2}{T_{X}\Teich(S)}$ to prove theorem \ref{conp}. In section \ref{5} we finish the proof of theorem \ref{hmrt}.\newline

\subsection*{Acknowledgments.}
This paper is part of the author's thesis. The author is greatly indebted to his advisor, Jeffrey Brock, for his consistent encouragement and support. He would like to thank George Daskalopoulos for introducing this problem to the author and for many suggestions and discussions to help finish this article. He also would like to thank Zheng Huang, Georg Schumacher, Mike Wolf, and Scott Wolpert for useful conversations and suggestions.

\section{Notations and Preliminaries}\label{np} 

\subsection{Surfaces}
Let $S$ be a closed surface of genus $g\geq 2$, and $\textsl{M}_{-1}$ denote the space of Riemannian metrics with constant curvature $-1$, and $X=(S,\sigma|dz|^2)$ be a particular element of $\textsl{M}_{-1}$. $Diff_{0}$, which is the group of diffeomorphisms 
isotopic to the identity, acts by pull-back on $\textsl{M}_{-1}$. The Teichm\"uller space of $S$ $T_{g}$ is defined by the quotient space
\begin{equation}
\nonumber M_{-1}/Diff_{0}.  
\end{equation}
The Teichm\"uller space has a natural complex structure, and its holomorphic cotangent space $T_{X}^{*}T_{g}$ is identified with the $\textit{quadratic  differentials}$ $Q(X)={\varphi(z)dz^{2}}$ on $X$. The \textit{Weil-Petersson metric} is the Hermitian
metric on $T_{g}$ arising from the the \textit{Petersson scalar  product}
\begin{equation}
 <\varphi,\psi>= \int_S \frac{\varphi \cdot \overline{\psi}}{\sigma^{2}}dzd\overline{z} \nonumber
\end{equation}
via duality. We will concern ourselves primarily with its Riemannian part $g_{WP}$. Throughout this paper, we denote the Teichm\"uller space endowed with the Weil-Petersson metric by $\Teich(S)$.\\ 

Setting $D=-2(\Delta-2)^{-1}$, where $\Delta$ is the Beltrami-Laplace operator on $X$, we have $D^{-1}=-\frac{1}{2}(\Delta-2)$. The following property has been proved in a lot of literature; for completeness, we still state the proof here.
\begin{proposition} 
Let $D$ be the operator above. Then \\
(1) $D$ is self-adjoint.\\
(2) $D$ is positive. 
\end{proposition}

\begin{proof}[Proof of (1)] Let $f$ and $g$ be two real-valued smooth functions on $X$, and $u=Df$, $v=Dg$. Then
\begin{eqnarray*}
\int_S Df\cdot gdA&=&\int_S u\cdot(-\frac{1}{2}(\Delta-2)v)dA=-\frac{1}{2} \int_S u\cdot (\Delta-2)vdA \\  
&=&-\frac{1}{2}\int_S v\cdot (\Delta-2)udA=\int_S Dg\cdot fdA,
\end{eqnarray*}
where the equality in the second row follows from the fact that $\Delta$ is self-adjoint on closed surfaces. For the case that $f$ and $g$ are complex-valued, one can prove it through the real and imaginary parts by using the same argument.\\

Proof of (2): Let $f$ be a real-valued smooth functions on $X$, and $u=Df$. Then
\begin{eqnarray*}  
\int_S Df\cdot f dA &=&\int_S u\cdot(-\frac{1}{2}(\Delta-2)u)dA=-\frac{1}{2} \int_S (u \cdot (\Delta u)-2u^2) dA \\ 
&=&\frac{1}{2}( \int_S |\nabla u|^2+ 2u^2 dA )\geq 0, \nonumber
\end{eqnarray*}
where the equality in the second row follows from  the Stoke's Theorem. The last equality holds if and only if $u=0$. That is, $D$ is positive. For the case that $f$ is complex-valued, one can show it by arguing the real and imaginary parts.
\end{proof}

For the Green function of the operator $-2(\Delta-2)^{-1}$, we have
\begin{proposition} \label{psfg}
Let $D$ be the operator above. Then there exists a Green function $G(w,z)$ for $D$ satisfying:\\
(1) $G(w,z)$ is positive.\\
(2) $G(w,z)$ is symmetric, i.e, $G(w,z)=G(z,w)$.

\end{proposition}

\begin{proof}
One can refer to \cite{Roelcke} and \cite{Wol86}.
\end{proof}

\textbf{The Riemannian tensor of the Weil-Petersson metric.} The curvature tensor is given by the following. Let $\mu_{\alpha},\mu_{\beta}$ be two elements in the tangent space at $X$, and 
\begin{eqnarray*}
g_{\alpha \overline{\beta}}=\int_X \mu_{\alpha} \cdot  \overline{\mu_{\beta}} dA, 
\end{eqnarray*}
where dA is the area element for X.  

Let us study the curvature tensor in these local coordinates. First of all, for the inverse of $(g_{i\overline{j}})$, we use the convention
\begin{eqnarray*}
g^{i\overline{j}} g_{k\overline{j}}=\delta_{ik}.
\end{eqnarray*}

The curvature tensor is given by
\begin{eqnarray*}
R_{i\overline{j}k\overline{l}}=\frac{\partial^2 g_{i\overline{j}}}{\partial t^{k}\partial \overline{t^{l}}}-g^{s\overline{t}}\frac{\partial g_{i\overline{t}}}{\partial t^{k}}\frac{\partial g_{s\overline{j}}}{\partial \overline{t^{l}}}.
\end{eqnarray*}

Since Ahlfors showed that the first derivatives of the metric tensor vanish at the base point $X$ in these coordinates, at $X$ we have
\begin{eqnarray}\label{cocf}
R_{i\overline{j}k\overline{l}}=\frac{\partial^2 g_{i\overline{j}}}{\partial t^{k}\partial \overline{t^{l}}}.
\end{eqnarray}

By the same argument in K\"ahler geometry we have
\begin{proposition}\label{bpfcoc}  
For any indices i, j, k, l, we have

\[(1) \quad R_{ij\overline{k}\overline{l}}=R_{\overline{i}\overline{j}kl}=0.\]

\[(2)\quad R_{i\overline{j}k\overline{l}}=-R_{i\overline{j}\overline{l}k}.\]

\[(3) \quad R_{i\overline{j}k\overline{l}}=R_{k\overline{j}i\overline{l}}.\]

\[(4) \quad R_{i\overline{j}k\overline{l}}=R_{i\overline{l}k\overline{j}}.\]

\end{proposition}
 
\begin{proof}
These follow from formula (\ref{cocf}) and the first Bianchi identity (one can refer to \cite{Jost}).
\end{proof}

Now let us state Wolpert's curvature formula, which is crucial in the proof of theorem \ref{conp}.

\begin{theorem}\label{cfow}(see \cite{Wol86}) 
The curvature tensor satisfies

\[R_{i\overline{j}k\overline{l}}=\int_{X} D(\mu_{i}\mu_{\overline{j}})\cdot (\mu_{k}\mu_{\overline{l}})dA+\int_{X} D(\mu_{i}\mu_{\overline{l}})\cdot (\mu_{k}\mu_{\overline{j}})dA.\]

\end{theorem}
 
\begin{definition}\label{ewcf}
Let $\mu_{*}$ be elements $\in T_{X}\Teich(S)$. Set
\begin{eqnarray*}
(i\overline{j},k\overline{l}):=\int_{X} D(\mu_{i}\mu_{\overline{j}})\cdot (\mu_{k}\mu_{\overline{l}})dA.
\end{eqnarray*}
\end{definition}

We close this section by rewriting theorem \ref{cfow} as follows.
\begin{theorem}\label{cfown}
\[R_{i\overline{j}k\overline{l}}=(i\overline{j},k\overline{l})+(i\overline{l},k\overline{j}).\]
\end{theorem}

\section{Curvature operator on subspaces of $\wedge{T_{X}^{2}\Teich(S)}$}\label{coos}

Before we study the curvature operator of $\Teich(S)$, let us set some neccessary notations. Let $U$ be a neighborhood of $X$ and $(t_{1},t_{2},...,t_{3g-3})$ be a local holomorphic coordinate on $U$, where $t_{i}=x_{i}+\textbf{i}y_{i} (1\leq i \leq 3g-3)$. Then $(x_{1},x_{2},...,x_{3g-3},y_{1},y_{2},...,y_{3g-3})$ is a real smooth coordinate in $U$. Furthermore, we have
\begin{eqnarray*}
\frac{\partial}{\partial x_{i}} =\frac{\partial}{\partial t_{i}}+\frac{\partial}{\partial \overline{t_{i}}}, \quad \ \ \frac{\partial}{\partial y_{i}} =\textbf{i}(\frac{\partial}{\partial t_{i}}-\frac{\partial}{\partial \overline{t_{i}}}).
\end{eqnarray*}

Let $T\Teich(S)$ be the real tangent bundle of $\Teich(S)$ and $\wedge^{2}T\Teich(S)$ be the exterior wedge product of $T\Teich(S)$ and itself. For any $X \in U$, we have

\[T_{X}\Teich(S)=Span\{\frac{\partial}{\partial x_{i}}(X),\frac{\partial}{\partial y_{j}}(X)\}_{1\leq i,j\leq3g-3}\] 

and

\[\wedge^{2}T\Teich(S)=Span\{\frac{\partial}{\partial x_{i}}\wedge \frac{\partial}{\partial x_{j}}, 
\frac{\partial}{\partial x_{k}}\wedge \frac{\partial}{\partial y_{l}}, \frac{\partial}{\partial y_{m}}\wedge \frac{\partial}{\partial y_{n}}\}.\]

Set
\begin{eqnarray*}
\wedge^{2}T_{X}^1\Teich(S):= Span\{\frac{\partial}{\partial x_{i}}\wedge \frac{\partial}{\partial x_{j}}\}, \\
\wedge^{2}T_{X}^2\Teich(S):= Span\{\frac{\partial}{\partial x_{k}}\wedge \frac{\partial}{\partial y_{l}}\}, \\
\wedge^{2}T_{X}^3\Teich(S):= Span\{\frac{\partial}{\partial y_{m}}\wedge \frac{\partial}{\partial y_{n}}\}.
\end{eqnarray*}

Hence, 
\begin{eqnarray*}
\wedge^{2}T_{X}\Teich(S)=Span\{\wedge^{2}T_{X}^1\Teich(S),\wedge^{2}T_{X}^2\Teich(S),\wedge^{2}T_{X}^3\Teich(S)\}.
\end{eqnarray*}

\subsection{The curvature operator on $\wedge^{2}T_{X}^1\Teich(S)$}
Let $\sum_{i j}{a_{ij}\frac{\partial}{\partial x_{i}}\wedge \frac{\partial}{\partial x_{j}}}$ be an element in $\wedge^{2}T_{X}^1\Teich(S)$, where $a_{ij}$ are real. Set
\begin{eqnarray*}
F(z,w)=\sum_{i,j=1}^{3g-3}{a_{ij}\mu_{i}(w)\cdot \overline{\mu_{j}(z)}}.
\end{eqnarray*}
The following proposition is influenced by theorem 4.1 in \cite{LSY}.
\begin{proposition} \label{coor}
Let $Q$ be the curvature operator and $D=-2(\Delta-2)^{-1}$, where $\Delta$ is the Beltrami-Laplace operator on $X$. $G$ is the Green function of $D$, and $\sum_{ij}a_{ij}\frac{\partial}{\partial x_{i}}\wedge \frac{\partial}{\partial x_{j}}$ is an element in $\wedge^{2}T_{X}^1\Teich(S)$, where $a_{ij}$ are real. Then we have
\begin{eqnarray*}
&&Q(\sum_{ij}a_{ij}\frac{\partial}{\partial x_{i}}\wedge \frac{\partial}{\partial x_{j}},\sum_{ij}a_{ij}\frac{\partial}{\partial x_{i}}\wedge \frac{\partial}{\partial x_{j}})\\
&=& \int_{X}{D(F(z,z)-\overline{F(z,z)})(F(z,z)-\overline{F(z,z)})}dA(z)\\
&-&2\cdot \int_{X\times X}G(z,w)|F(z,w))|^2dA(w)dA(z)\\
&+&2 \cdot \Re\{\int_{X\times X}G(z,w)F(z,w)F(w,z)dA(w)dA(z)\},
\end{eqnarray*}
where $F(z,w)=\sum_{i,j=1}^{3g-3}{a_{ij}\mu_{i}(w)\cdot \overline{\mu_{j}(z)}}$.
\end{proposition}

\begin{proof} Since $\frac{\partial}{\partial x_{i}} =\frac{\partial}{\partial t_{i}}+\frac{\partial}{\partial \overline{t_{i}}}$, from proposition \ref{bpfcoc},
\begin{eqnarray*}
&& Q(\sum_{ij}a_{ij}\frac{\partial}{\partial x_{i}}\wedge \frac{\partial}{\partial x_{j}},\sum_{ij}a_{ij}\frac{\partial}{\partial x_{i}}\wedge \frac{\partial}{\partial x_{j}})\\
&=& \sum_{i,j,k,l}a_{ij}a_{kl}(R_{i\overline{j}k\overline{l}}+R_{i\overline{j}\overline{k}l}+R_{\overline{i}jk\overline{l}}+R_{\overline{i}j\overline{k}l})\\
&=& \sum_{i,j,k,l}a_{ij}a_{kl}(R_{i\overline{j}k\overline{l}}-R_{i\overline{j}l\overline{k}}-R_{j\overline{i}k\overline{l}}+R_{j\overline{i}l\overline{k}})\\
&=& \sum_{i,j,k,l}a_{ij}a_{kl}((i\overline{j},k\overline{l})+(i\overline{l},k\overline{j})-(i\overline{j},l\overline{k})-(i\overline{k},l\overline{j})\\
&-&(j\overline{i},k\overline{l})-(j\overline{l},k\overline{i})+(j\overline{i},l\overline{k})+(j\overline{k},l\overline{i})) \ \ (\text{by theorem 2.6})\\
&=& \sum_{i,j,k,l}a_{ij}a_{kl}(i\overline{j}-j\overline{i},k\overline{l}-l\overline{k})\\
&+& \sum_{i,j,k,l}a_{ij}a_{kl}((i\overline{l},k\overline{j})+(l\overline{i},j\overline{k}))\\
&-& \sum_{i,j,k,l}a_{ij}a_{kl}((i\overline{k},l\overline{j})+(j\overline{l},k\overline{i})).
\end{eqnarray*} 

For the first term, from definition \ref{ewcf},
\begin{eqnarray*}
&& \sum_{i,j,k,l}a_{ij}a_{kl}(i\overline{j}-j\overline{i},k\overline{l}-l\overline{k})\\
&=& \int_{X}D(\sum_{ij}a_{ij}\mu_{i}\overline{\mu_{j}} -\sum_{ij}a_{ij}\mu_{j}\overline{\mu_{i}})
(\sum_{ij}a_{ij}\mu_{i}\overline{\mu_{j}}-\sum_{ij}a_{ij}\mu_{j}\overline{\mu_{i}})dA(z)\\
&=& \int_{X}{D(F(z,z)-\overline{F(z,z)})(F(z,z)-\overline{F(z,z)})}dA(z).
\end{eqnarray*}

For the second term, after applying the Green function $G$ we have
\begin{small}
\begin{eqnarray*}
&& \sum_{i,l}a_{ij}a_{kl}((i\overline{l},k\overline{j})+(l\overline{i},j\overline{k}))=2 \cdot \Re\{\sum_{i,l}a_{ij}a_{kl}((i\overline{l},k\overline{j})\}\\
&=& 2\cdot \Re\{\int_{X} D(\sum_{i}a_{ij}\mu_{i}\overline{\mu_{l}})(\sum_{k}a_{kl}\mu_{k}\overline{\mu_{j}})dA(z)\}\\
&=& 2\cdot \Re\{\int_{X}\int_{X}G(w,z)\sum_{i}a_{ij}\mu_{i}(w)\overline{\mu_{l}(w)}(\sum_{k}a_{kl}\mu_{k}(z)\overline{\mu_{j}}(z))dA(z)dA(w)\}.
\end{eqnarray*}
\end{small}
From the definition of $F(z,w)$,
\begin{eqnarray*}
&& \sum_{i,j,k,l}a_{ij}a_{kl}((i\overline{l},k\overline{j})+(l\overline{i},j\overline{k})) \\
&=& 2\cdot \Re\{\int_{X\times X}G(z,w)F(z,w)F(w,z)dA(w)dA(z)\}.
\end{eqnarray*}

For the last term, we use an argument similar to that for the second term,
\begin{small}
\begin{eqnarray*}
&& \sum_{i,k}a_{ij}a_{kl}((i\overline{k},l\overline{j})+(k\overline{i},j\overline{l}))=2 \cdot \Re\{\sum_{i,k}a_{ij}a_{kl}((i\overline{k},l\overline{j})\}\\
&=& 2\cdot \Re\{\int_{X} D(\sum_{i}a_{ij}\mu_{i}\overline{\sum_{k}a_{kl}\mu_{k}})(\mu_{l}\overline{\mu_{j}})dA(z)\}\\
&=& 2\cdot \Re\{\int_{X}\int_{X}G(w,z)\sum_{i}a_{ij}\mu_{i}(w)\overline{\sum_{k}a_{kl}\mu_{k}(w)}(\mu_{l}(z)\overline{\mu_{j}}(z))dA(z)dA(w)\}
\end{eqnarray*}
\end{small}
From the definition of $F(z,w)$,
\begin{eqnarray*}
&& \sum_{i,j,k,l}a_{ij}a_{kl}((i\overline{k},l\overline{j})+(k\overline{i},j\overline{l})) \\
&=& 2\cdot \Re\{\int_{X\times X}G(z,w)F(z,w)\overline{F(z,w)}dA(w)dA(z)\}\\
&=& 2\cdot \int_{X\times X}G(z,w)|F(z,w)|^2dA(w)dA(z).
\end{eqnarray*}
The conclusion follows from combining the three terms above.
\end{proof}

Using the Green function's positivity and symmetry,
\begin{theorem} \label{coorn}
Under the same conditions in proposition \ref{coor}, $Q$ is negative definite on $\wedge^{2}T_{X}^1\Teich(S)$.
\end{theorem}

\begin{proof} By proposition \ref{coor} we have
\begin{eqnarray*}
&&Q(\sum_{ij}a_{ij}\frac{\partial}{\partial x_{i}}\wedge \frac{\partial}{\partial x_{j}},\sum_{ij}a_{ij}\frac{\partial}{\partial x_{i}}\wedge \frac{\partial}{\partial x_{j}})\\
&=&\int_{X}{D(F(z,z)-\overline{F(z,z)})(F(z,z)-\overline{F(z,z)})}dA(z) \\
&-&2\cdot (\int_{X\times X}G(z,w)|F(z,w))|^2dA(w)dA(z)\\
&-&\Re\{\int_{X\times X}G(z,w)F(z,w)F(w,z)dA(w)dA(z)\}).
\end{eqnarray*}
For the first term, since $F(z,z)-\overline{F(z,z)}=2\textbf{i}\Im \{F(z,z)\}$, by the positivity of the operator $D$,
\begin{eqnarray*}
&& \int_{X}{D(F(z,z)-\overline{F(z,z)})(F(z,z)-\overline{F(z,z)})}dA(z)\\
&=&-4\cdot \int_{X}D({\Im \{F(z,z)\})(\Im \{F(z,z)\})}dA(z) \leq 0.
\end{eqnarray*}
 
For the second term, using the Cauchy-Schwarz inequality, 
\begin{eqnarray*}
  &&|\int_{X\times X}G(z,w)F(z,w)F(w,z)dA(w)dA(z)| \\
 &\leq& \int_{X\times X}|G(z,w)F(z,w)F(w,z)|dA(w)dA(z)\\
 &\leq& \sqrt{ \int_{X\times X}|G(z,w)||F(z,w)|^2dA(w)dA(z)}\\
 &\times& \sqrt{ \int_{X\times X}|G(z,w)||F(w,z)|^2dA(w)dA(z)}\\
 &=&\int_{X\times X}G(z,w)|F(z,w)|^2dA(w)dA(z),
\end{eqnarray*}
since $G$ is positive and symmetric (see proposition \ref{psfg}).

Combining these three terms, we get that $Q$ is non-positive definite on $\wedge^{2}T_{X}^1\Teich(S)$.

Furthermore, equality holds precisely when
\begin{eqnarray*}
Q(\sum_{ij}a_{ij}\frac{\partial}{\partial x_{i}}\wedge \frac{\partial}{\partial x_{j}},
\sum_{ij}a_{ij}\frac{\partial}{\partial x_{i}}\wedge \frac{\partial}{\partial x_{j}})=0;
\end{eqnarray*}
that is, there exists a constant complex number $k$ such that both of the following hold:
\begin{eqnarray*}
\begin{cases}
F(z,z)=\overline{F(z,z)},\\
F(z,w)=k\cdot \overline{F(w,z)}.
\end{cases}
\end{eqnarray*}
If we let $z=w$, we get $k=1$. Hence, the last equation is equivalent to
\begin{eqnarray*}
\sum_{ij}(a_{ij}-a_{ji})\mu_{i}(w)\overline{\mu_{j}}(z)=0.
\end{eqnarray*}
Since $\{\mu_{i}\}_{i\geq 1}$ is a basis,
\begin{eqnarray*}
a_{ij}=a_{ji}.
\end{eqnarray*}
This means $\sum_{ij}a_{ij}\frac{\partial}{\partial x_{i}}\wedge \frac{\partial}{\partial x_{j}}=0$. That is, $Q$ is negative definite on $\wedge^{2}T_{X}^1\Teich(S)$.  
\end{proof}

\subsection{The curvature operator on $\wedge^{2}T_{X}^2\Teich(S)$} 

Let $b_{ij}$ be real and $\sum_{i j}{b_{ij}\frac{\partial}{\partial x_{i}}\wedge \frac{\partial}{\partial y_{j}}} \in \wedge^{2}T_{X}^2\Teich(S)$. Set
\begin{eqnarray*}
H(z,w)=\sum_{i,j=1}^{3g-3}b_{ij}\mu_{i}(w)\cdot \overline{\mu_{j}(z)}.
\end{eqnarray*}

Using a similar computation in proposition \ref{coor}, the formula for the curvature operator on $\wedge^{2}T_{X}^2\Teich(S)$ is given as follows.
\begin{proposition}\label{cooxy} 
Let $Q$ be the curvature operator and $D$ be the same operator as shown in proposition \ref{coor}. Let $\sum_{ij}b_{ij}\frac{\partial}{\partial x_{i}}\wedge \frac{\partial}{\partial y_{j}}$ be an element in $\wedge^{2}T_{X}^2\Teich(S)$, where $b_{ij}$ are real.
Then we have
\begin{eqnarray*}
&&Q(\sum_{ij}b_{ij}\frac{\partial}{\partial x_{i}}\wedge \frac{\partial}{\partial y_{j}},\sum_{ij}b_{ij}\frac{\partial}{\partial x_{i}}\wedge \frac{\partial}{\partial y_{j}})\\
&=& -\int_{X}{D(H(z,z)+\overline{H(z,z)})(H(z,z)+\overline{H(z,z)})}dA(z) \\
&-&2\cdot \int_{X\times X}G(z,w)|H(z,w))|^2dA(w)dA(z)\\
&-&2\cdot \Re\{\int_{X\times X}G(z,w)H(z,w)H(w,z)dA(w)dA(z),\}
\end{eqnarray*}
where $H(z,w)=\sum_{i,j=1}^{3g-3}{b_{ij}\mu_{i}(w)\cdot \overline{\mu_{j}(z)}}$.
\end{proposition}

\begin{proof}
Since $\frac{\partial}{\partial x_{i}} =\frac{\partial}{\partial t_{i}}+\frac{\partial}{\partial \overline{t_{i}}}$ and $\frac{\partial}{\partial y_{i}} =\textbf{i}(\frac{\partial}{\partial t_{i}}-\frac{\partial}{\partial \overline{t_{i}}})$, from proposition \ref{bpfcoc},
\begin{eqnarray*}
&& Q(\sum_{ij}b_{ij}\frac{\partial}{\partial x_{i}}\wedge \frac{\partial}{\partial y_{j}},\sum_{ij}b_{ij}\frac{\partial}{\partial x_{i}}\wedge \frac{\partial}{\partial y_{j}})\\
&=& -\sum_{i,j,k,l}b_{ij}b_{kl}(R_{i\overline{j}k\overline{l}}-R_{i\overline{j}\overline{k}l}-R_{\overline{i}jk\overline{l}}+R_{\overline{i}j\overline{k}l})\\
&=& -\sum_{i,j,k,l}b_{ij}b_{kl}(R_{i\overline{j}k\overline{l}}+R_{i\overline{j}l\overline{k}}+R_{j\overline{i}k\overline{l}}+R_{j\overline{i}l\overline{k}}).\\
&=& -\sum_{i,j,k,l}b_{ij}b_{kl}((i\overline{j},k\overline{l})+(i\overline{l},k\overline{j})+(i\overline{j},l\overline{k})+(i\overline{k},l\overline{j}) \ \ (\text{by theorem \ref{cfown}}) \\
&+&(j\overline{i},k\overline{l})+(j\overline{l},k\overline{i})+(j\overline{i},l\overline{k})+(j\overline{k},l\overline{i}))\\
&=& -\sum_{i,j,k,l}b_{ij}b_{kl}(i\overline{j}+j\overline{i},k\overline{l}+l\overline{k})\\
&-& \sum_{i,j,k,l}b_{ij}b_{kl}((i\overline{l},k\overline{j})+(l\overline{i},j\overline{k}))\\
&-& \sum_{i,j,k,l}b_{ij}b_{kl}((i\overline{k},l\overline{j})+(j\overline{l},k\overline{i})).
\end{eqnarray*} 
For the first term, from definition \ref{ewcf},
\begin{eqnarray*}
&& -\sum_{i,j,k,l}b_{ij}b_{kl}(i\overline{j}+j\overline{i},k\overline{l}+l\overline{k})\\
&=& -\int_{X}D(\sum_{ij}b_{ij}\mu_{i}\overline{\mu_{j}}+\sum_{ij}b_{ij}\mu_{j}\overline{\mu_{i}})
(\sum_{ij}b_{ij}\mu_{i}\overline{\mu_{j}}+\sum_{ij}b_{ij}\mu_{j}\overline{\mu_{i}})dA(z)\\
&=& -\int_{X}{D(H(z,z)+\overline{H(z,z)})(H(z,z)+\overline{H(z,z)})}dA(z).
\end{eqnarray*}
For the second term and the third term, using the same argument in the proof of proposition \ref{coor}, we have
\begin{eqnarray*}
&&\sum_{i,j,k,l}b_{ij}b_{kl}((i\overline{l},k\overline{j})+(l\overline{i},j\overline{k}))
\\&=&2\cdot \Re\{\int_{X\times X}G(z,w)H(z,w)H(w,z)dA(w)dA(z)\}
\end{eqnarray*}
and
\begin{eqnarray*}
&&\sum_{i,j,k,l}b_{ij}b_{kl}((i\overline{k},l\overline{j})+(k\overline{i},j\overline{l}))\\
&=&2\cdot \int_{X\times X}G(z,w)|H(z,w)|^2dA(w)dA(z).
\end{eqnarray*}
Combining these three terms we get the proposition.
\end{proof}

Using the same method in theorem \ref{coorn}, one can prove the following non-positivity result.

\begin{theorem} \label{cooxyn}
Under the same conditions of proposition \ref{cooxy}, then $Q$ is non-positive definite on $\wedge^{2}T_{X}^2\Teich(S)$, and the zero level subsets of $Q(\cdot,\cdot)$ are $\{\sum_{ij}b_{ij}\frac{\partial}{\partial x_{i}}\wedge \frac{\partial}{\partial y_{j}};\ b_{ij}=-b_{ji}\}$.
\end{theorem}

\begin{proof}
Let $\sum_{ij}b_{ij}\frac{\partial}{\partial x_{i}}\wedge \frac{\partial}{\partial y_{j}}$ be an element in $\wedge^{2}T_{X}^2\Teich(S)$. From proposition \ref{cooxy} we have
\begin{eqnarray*}
&&Q(\sum_{ij}b_{ij}\frac{\partial}{\partial x_{i}}\wedge \frac{\partial}{\partial y_{j}},\sum_{ij}b_{ij}\frac{\partial}{\partial x_{i}}\wedge \frac{\partial}{\partial y_{j}})\\
&=&-\int_{X}{D(H(z,z)+\overline{H(z,z)})(H(z,z)+\overline{H(z,z)})}dA(z) \\
&-&2(\cdot \int_{X\times X}G(z,w)|H(z,w))|^2dA(w)dA(z)\\
&+&\cdot \Re\{\int_{X\times X}G(z,w)H(z,w)H(w,z)dA(w)dA(z)\}).
\end{eqnarray*}
For the first term, since $H(z,z)+\overline{H(z,z)}=2\cdot \Re\{H(z,z)\}$, by the positivity of the operator $D$,
\begin{eqnarray*}
&&-\int_{X}{D(H(z,z)+\overline{H(z,z)})(H(z,z)+\overline{H(z,z)})}dA(z)\\
&=&-4\int_{X}{D(\Re\{H(z,z)\})(\Re\{H(z,z)\})}dA(z)\leq 0.
\end{eqnarray*}
 
For the second term, using the Cauchy-Schwarz inequality, 
\begin{eqnarray*}
  &&|\int_{X\times X}G(z,w)H(z,w)H(w,z)dA(w)dA(z)| \\
 &\leq& \int_{X\times X}|G(z,w)H(z,w)H(w,z)|dA(w)dA(z)\\
 &\leq& \sqrt{ \int_{X\times X}|G(z,w)||H(z,w)|^2dA(w)dA(z)}\\
 &\times&\sqrt{ \int_{X\times X}|G(z,w)||H(w,z)|^2dA(w)dA(z)}\\
 &=& \int_{X\times X}G(z,w)|H(z,w)|^2dA(w)dA(z),
\end{eqnarray*}
since $G$ is positive and symmetric.

Combining these two terms, we get $Q$ is non-positive on $\wedge^{2}T_{X}^2\Teich(S)$.

Using the same argument as in the proof of theorem \ref{coorn}, 
\begin{eqnarray*}
Q(\sum_{ij}b_{ij}\frac{\partial}{\partial x_{i}}\wedge \frac{\partial}{\partial y_{j}},\sum_{ij}b_{ij}\frac{\partial}{\partial x_{i}}\wedge \frac{\partial}{\partial y_{j}})=0
\end{eqnarray*}
if and only if there exists a constant complex number $k$ such that both of the following hold:
\begin{eqnarray*}
\begin{cases}
H(z,z)=-\overline{H(z,z)},\\
H(z,w)=k\cdot \overline{H(w,z)}.
\end{cases}
\end{eqnarray*}

If we let $z=w$, we get $k=-1$. Hence, the last equation is equivalent to
\begin{eqnarray*}
\sum_{ij}(b_{ij}+b_{ji})\mu_{i}(w)\overline{\mu_{j}}(z)=0.
\end{eqnarray*} 
Since $\{\mu_{i}\}_{i\geq 1}$ is a basis,
\begin{eqnarray*}
b_{ij}=-b_{ji}.
\end{eqnarray*}
\end{proof}

\subsection{The curvature operator on $\wedge^{2}T_{X}^3\Teich(S)$}
Let $\textbf{J}$ be the almost complex structure on $\Teich(S)$. Since $\{t_{i}\}$ is a holomorphic coordinate, we have
\begin{eqnarray*}
\textbf{J}\frac{\partial}{\partial x_{i}}=\frac{\partial}{\partial y_{i}} \quad  \textbf{J}\frac{\partial}{\partial y_{i}}=-\frac{\partial}{\partial x_{i}}.
\end{eqnarray*}
Since the Weil-Petersson metric is a K\"ahler metric, $\textbf{J}$ is an isometry on the tangent space. In particular we have
\begin{eqnarray*}
R(V_{1},V_{2},V_{3},V_{4})&=&R(\textbf{J}V_{1},\textbf{J}V_{2},\textbf{J}V_{3},\textbf{J}V_{4})\\
&=&R(\textbf{J}V_{1},\textbf{J}V_{2},V_{3},V_{4})=R(V_{1},V_{2},\textbf{J}V_{3},\textbf{J}V_{4}),
\end{eqnarray*} 
where $R$ is the curvature tensor and $V_{i}$ are real tangent vectors in $T_{X}\Teich(S)$. Once can refer to \cite{Jost} for more details.

Let $C=\sum_{ij}{c_{ij}\frac{\partial}{\partial y_{i}}\wedge \frac{\partial}{\partial y_{j}}}$ be an element in 
$\wedge^{2}T_{X}^3\Teich(S)$, where $c_{ij}$ are real. Set
\begin{eqnarray*}
K(z,w)=\sum_{i,j=1}^{3g-3}{c_{ij}\mu_{i}(w)\cdot\overline{\mu_{j}(z)}}.
\end{eqnarray*}
\begin{proposition}\label{cfot3}
Let $Q$ be the curvature operator, and $\sum_{ij}c_{ij}\frac{\partial}{\partial y_{i}}\wedge \frac{\partial}{\partial y_{j}}$ be an element in $\wedge^{2}T_{X}^3\Teich(S)$.
Then we have
\begin{eqnarray*}
&&Q(\sum_{ij}c_{ij}\frac{\partial}{\partial y_{i}}\wedge \frac{\partial}{\partial y_{j}},\sum_{ij}c_{ij}\frac{\partial}{\partial y_{i}}\wedge \frac{\partial}{\partial y_{j}})\\
&=&\int_{X}{D(K(z,z)-\overline{K(z,z)})(K(z,z)-\overline{K(z,z)})}dA(z) \\
&-&2\cdot \int_{X\times X}G(z,w)|K(z,w))|^2dA(w)dA(z)\\
&+&2\cdot \Re\{\int_{X\times X}G(z,w)K(z,w)K(w,z)dA(w)dA(z)\}.
\end{eqnarray*}
\end{proposition}

\begin{proof}
Since $\frac{\partial}{\partial y_{i}} =\textbf{J}\frac{\partial}{\partial x_{i}}=\textbf{J}\frac{\partial}{\partial t_{i}}+\textbf{J}\frac{\partial}{\partial \overline{t_{i}}}$ and \textbf{J} is an 
isometry, by proposition \ref{bpfcoc},
\begin{eqnarray*}
&& Q(\sum_{ij}c_{ij}\frac{\partial}{\partial y_{i}}\wedge \frac{\partial}{\partial y_{j}},\sum_{ij}c_{ij}\frac{\partial}{\partial y_{i}}\wedge \frac{\partial}{\partial y_{j}})\\
&=& \sum_{i,j,k,l}c_{ij}c_{kl}(R_{i\overline{j}k\overline{l}}+R_{i\overline{j}\overline{k}l}+R_{\overline{i}jk\overline{l}}+R_{\overline{i}j\overline{k}l})\\
&=& Q(\sum_{ij}c_{ij}\frac{\partial}{\partial x_{i}}\wedge \frac{\partial}{\partial x_{j}},\sum_{ij}c_{ij}\frac{\partial}{\partial x_{i}}\wedge \frac{\partial}{\partial x_{j}}).
\end{eqnarray*}
By proposition \ref{coor},
\begin{eqnarray*}
&&Q(\sum_{ij}c_{ij}\frac{\partial}{\partial y_{i}}\wedge \frac{\partial}{\partial y_{j}},\sum_{ij}c_{ij}\frac{\partial}{\partial y_{i}}\wedge \frac{\partial}{\partial y_{j}})\\
&=& \int_{X}{D(K(z,z)-\overline{K(z,z)})(K(z,z)-\overline{K(z,z)})}dA(z)\\
&-&2\cdot\int_{X\times X}G(z,w)|K(z,w))|^2dA(w)dA(z)\\
&+&2\cdot \Re\{\int_{X\times X}G(z,w)K(z,w)K(w,z)dA(w)dA(z)\}.
\end{eqnarray*}
\end{proof}

Using the same argument as in the proof of theorem \ref{coorn} we have
\begin{theorem}\label{conot3} 
Let $Q$ be the curvature operator as above, then $Q$ is a negative definite operator on $\wedge^{2}T_{X}^3\Teich(S)$.
\end{theorem}

\section{Curvature operator on $\wedge^{2}T_{X}\Teich(S)$}\label{4}
Every element in $\wedge^{2}T_{X}\Teich(S)$ can 
be represented by $\sum_{ij}(a_{ij}\frac{\partial}{\partial x_{i}}\wedge \frac{\partial}{\partial x_{j}}+b_{ij}\frac{\partial}{\partial x_{i}}\wedge \frac{\partial}{\partial y_{j}}+ c_{ij}\frac{\partial}{\partial y_{i}}\wedge \frac{\partial}{\partial y_{j}})$.  

\begin{proposition} \label{corf}
Let $Q$ be the curvature operator. Then
\begin{small}
\begin{eqnarray*}
&&Q(\sum_{ij}a_{ij}\frac{\partial}{\partial x_{i}}\wedge \frac{\partial}{\partial x_{j}}+b_{ij}\frac{\partial}{\partial x_{i}}\wedge \frac{\partial}{\partial y_{j}}+ c_{ij}\frac{\partial}{\partial y_{i}}\wedge \frac{\partial}{\partial y_{j}},\\
&&\sum_{ij}a_{ij}\frac{\partial}{\partial x_{i}}\wedge \frac{\partial}{\partial x_{j}}+b_{ij}\frac{\partial}{\partial x_{i}}\wedge \frac{\partial}{\partial y_{j}}+ c_{ij}\frac{\partial}{\partial y_{i}}\wedge \frac{\partial}{\partial y_{j}})\\
&=&Q(\sum_{ij}d_{ij}\frac{\partial}{\partial x_{i}}\wedge \frac{\partial}{\partial x_{j}}+b_{ij}\frac{\partial}{\partial x_{i}}\wedge \frac{\partial}{\partial y_{j}},\sum_{ij}d_{ij}\frac{\partial}{\partial x_{i}}\wedge \frac{\partial}{\partial x_{j}}+b_{ij}\frac{\partial}{\partial x_{i}}\wedge \frac{\partial}{\partial y_{j}}), 
\end{eqnarray*} 
\end{small}
where $d_{ij}=a_{ij}+c_{ij}$.
\end{proposition}

\begin{proof}
Since the almost complex structure $\textbf{J}$ is an isometry and $\textbf{J}\frac{\partial}{\partial x_{i}}=\frac{\partial}{\partial y_{i}}$, we have 
\begin{eqnarray*}
Q(\frac{\partial}{\partial x_{i}}\wedge \frac{\partial}{\partial x_{j}},\frac{\partial}{\partial y_{i}}\wedge \frac{\partial}{\partial y_{j}})&=&R(\frac{\partial}{\partial x_{i}}, \frac{\partial}{\partial x_{j}},\textbf{J}\frac{\partial}{\partial x_{i}}, \textbf{J}\frac{\partial}{\partial x_{j}})\\
&=&R(\frac{\partial}{\partial x_{i}}, \frac{\partial}{\partial x_{j}},\frac{\partial}{\partial x_{i}}, \frac{\partial}{\partial x_{j}})
\end{eqnarray*} 

and 

\begin{eqnarray*}
Q(\frac{\partial}{\partial x_{i}}\wedge \frac{\partial}{\partial y_{j}},\frac{\partial}{\partial y_{i}}\wedge \frac{\partial}{\partial y_{j}})&=&R(\frac{\partial}{\partial x_{i}}, \frac{\partial}{\partial y_{j}},\textbf{J}\frac{\partial}{\partial x_{i}}, \textbf{J}\frac{\partial}{\partial x_{j}})\\
&=&R(\frac{\partial}{\partial x_{i}}, \frac{\partial}{\partial y_{j}},\frac{\partial}{\partial x_{i}}, \frac{\partial}{\partial x_{j}}).
\end{eqnarray*} 

The conclusion follows by expanding $Q$ and applying the two equations above.
\end{proof}

If one wants to determine whether the curvature operator $Q$ is non-positive definite on $\wedge^{2}T_{X}\Teich(S)$, by proposition \ref{corf} it is sufficient to see if $Q$ is non-positive definite on $Span\{\wedge^{2}T_{X}^{1}\Teich(S),\wedge^{2}T_{X}^{2}\Teich(S)\}$. 
\begin{proposition}\label{cofon12}
Let $Q$ be the curvature operator, $\sum_{ij}a_{ij}\frac{\partial}{\partial x_{i}}\wedge \frac{\partial}{\partial x_{j}}$ be an element in $\wedge^{2}T_{X}^{1}\Teich(S)$, and $\sum_{ij}b_{ij}\frac{\partial}{\partial x_{i}}\wedge \frac{\partial}{\partial y_{j}}$ be an element in $\wedge^{2}T_{X}^{2}\Teich(S)$. Then we have
\begin{eqnarray*}
&&Q(\sum_{ij}a_{ij}\frac{\partial}{\partial x_{i}}\wedge \frac{\partial}{\partial x_{j}},\sum_{ij}b_{ij}\frac{\partial}{\partial x_{i}}\wedge \frac{\partial}{\partial y_{j}})\\
&=&\textbf{i} \cdot \int_{X}{D(F(z,z)-\overline{F(z,z)})\cdot(H(z,z)+\overline{H(z,z)})}dA(z) \\
&-&2\cdot \Im\{\int_{X\times X}G(z,w)F(z,w)\overline{H(z,w)})dA(w)dA(z)\}\\
&-&2\cdot \Im\{\int_{X\times X}G(z,w)F(z,w)H(w,z)dA(w)dA(z)\}.
\end{eqnarray*}
\end{proposition}

\begin{proof}
Since $\frac{\partial}{\partial x_{i}} =\frac{\partial}{\partial t_{i}}+\frac{\partial}{\partial \overline{t_{i}}}$ and $\frac{\partial}{\partial y_{i}} =\textbf{i}(\frac{\partial}{\partial t_{i}}-\frac{\partial}{\partial \overline{t_{i}}})$, by proposition \ref{bpfcoc},
\begin{eqnarray*}
&& Q(\sum_{ij}a_{ij}\frac{\partial}{\partial x_{i}}\wedge \frac{\partial}{\partial x_{j}},\sum_{ij}b_{ij}\frac{\partial}{\partial x_{i}}\wedge \frac{\partial}{\partial y_{j}})\\
&=&(-\textbf{i})\sum_{i,j,k,l}a_{ij}b_{kl}(-R_{i\overline{j}k\overline{l}}+R_{i\overline{j}\overline{k}l}-R_{\overline{i}jk\overline{l}}+R_{\overline{i}j\overline{k}l})\\
&=&(-\textbf{i})\sum_{i,j,k,l}a_{ij}b_{kl}(-R_{i\overline{j}k\overline{l}}-R_{i\overline{j}l\overline{k}}+R_{j\overline{i}k\overline{l}}+R_{j\overline{i}l\overline{k}})\\
&=&(-\textbf{i})\sum_{i,j,k,l}a_{ij}b_{kl}(-(i\overline{j},k\overline{l})-(i\overline{l},k\overline{j})-(i\overline{j},l\overline{k})-(i\overline{k},l\overline{j})\\
&+&(j\overline{i},k\overline{l})+(j\overline{l},k\overline{i})+(j\overline{i},l\overline{k})+(j\overline{k},l\overline{i})) \ \ (\text{by theorem \ref{cfown}})\\
&=&(-\textbf{i})\sum_{i,j,k,l}a_{ij}b_{kl}(j\overline{i}-i\overline{j},k\overline{l}+l\overline{k})\\
&+& (-\textbf{i})\sum_{i,j,k,l}a_{ij}b_{kl}(-(i\overline{l},k\overline{j})+(l\overline{i},j\overline{k}))\\
&+& (-\textbf{i})\sum_{i,j,k,l}a_{ij}b_{kl}(-(i\overline{k},l\overline{j})+(j\overline{l},k\overline{i})).
\end{eqnarray*} 

For the first term, by definition \ref{ewcf},
\begin{eqnarray*}
&& \sum_{i,j,k,l}a_{ij}b_{kl}(j\overline{i}-i\overline{j},k\overline{l}+l\overline{k})\\
&=& \int_{X}D(\sum_{ij}a_{ij}\overline{\mu_{i}}\mu_{j} -\sum_{ij}a_{ij}\mu_{i}\overline{\mu_{j}})
(\sum_{kl}b_{kl}\mu_{k}\overline{\mu_{l}}+\sum_{kl}b_{kl}\mu_{l}\overline{\mu_{k}})dA(z)\\
&=& \int_{X}{D(\overline{F(z,z)}-F(z,z))(H(z,z)+\overline{H(z,z)})}dA(z).
\end{eqnarray*}

For the second term, by using the Green function $G$ of $D$,
\begin{small}
\begin{eqnarray*}
&& \sum_{i,l}a_{ij}b_{kl}(-(i\overline{l},k\overline{j})+(l\overline{i},j\overline{k}))\\
&=&2\textbf{i}\cdot \Im\{\sum_{i,l}a_{ij}b_{kl}(-(i\overline{l},k\overline{j})\}\\
&=&-2\textbf{i}\cdot \Im\{\int_{X} D(\sum_{i}a_{ij}\mu_{i}\overline{\mu_{l}})(\sum_{k}b_{kl}\mu_{k}\overline{\mu_{j}})dA(z)\}\\
&=&-2\textbf{i}\cdot \Im\{\int_{X}\int_{X}G(w,z)\sum_{i}a_{ij}\mu_{i}(w)\overline{\mu_{l}(w)}(\sum_{k}b_{kl}\mu_{k}(z)\overline{\mu_{j}}(z))dA(z)dA(w)\}\\ 
&=&-2\textbf{i}\cdot \Im\{\int_{X\times X}G(z,w)F(z,w)H(w,z)dA(w)dA(z)\}.
\end{eqnarray*}
\end{small}

For the last term, 
\begin{small}
\begin{eqnarray*}
&& \sum_{i,k}a_{ij}b_{kl}(-(i\overline{k},l\overline{j})+(k\overline{i},j\overline{l}))=-2\textbf{i}\cdot \Im\{\sum_{i,k}a_{ij}b_{kl}(i\overline{k},l\overline{j})\}\\
&=& -2\textbf{i}\cdot \Im\{\int_{X} D(\sum_{i}a_{ij}\mu_{i}\overline{\sum_{k}b_{kl}\mu_{k}})(\mu_{l}\overline{\mu_{j}})dA(z)\}\\
&=&-2\textbf{i}\cdot \Im\{\int_{X}\int_{X}G(w,z)\sum_{i}a_{ij}\mu_{i}(w)\overline{\sum_{k}b_{kl}\mu_{k}(w)}(\mu_{l}(z)\overline{\mu_{j}}(z))dA(z)dA(w)\}\\
&=&-2\textbf{i}\cdot \Im\{\int_{X\times X}G(z,w)F(z,w)\overline{H(z,w)}dA(w)dA(z)\}.
\end{eqnarray*}
\end{small}
Combining these three terms above, we get the lemma.
\end{proof}

The following proposition will give the formula for curvature operator $Q$ on $Span\{\wedge^{2}T_{X}^{1}\Teich(S), \wedge^{2}T_{X}^{2}\Teich(S)\}$. 
Setting
\begin{eqnarray*}
A=\sum_{ij}a_{ij}\frac{\partial}{\partial x_{i}}\wedge \frac{\partial}{\partial x_{j}}, \quad \quad B=\sum_{ij}b_{ij}\frac{\partial}{\partial x_{i}}\wedge \frac{\partial}{\partial y_{j}},
\end{eqnarray*}
on $Span\{\wedge^{2}T_{X}^1\Teich(S),\wedge^{2}T_{X}^2\Teich(S)\}$, we have 

\begin{proposition} \label{cooxym}
Let $Q$ be the curvature operator on $\Teich(S)$. Let $A=\sum_{ij}a_{ij}\frac{\partial}{\partial x_{i}}\wedge \frac{\partial}{\partial x_{j}}$ and $B=\sum_{ij}b_{ij}\frac{\partial}{\partial x_{i}}\wedge \frac{\partial}{\partial y_{j}}$. Then we have

\begin{small}
\begin{eqnarray*}
&&Q(A+B,A+B)=\\
&-4&\int_{X}{D (\Im\{F(z,z)+\textbf{i}H(z,z)\})\cdot(\Im\{F(z,z)+\textbf{i}H(z,z)\})dA(z)}\\
&-2&\int_{X\times X}G(z,w)|F(z,w)+\textbf{i}H(z,w))|^2dA(w)dA(z)\\
&+2&\Re\{\int_{X\times X}G(z,w)(F(z,w)+\textbf{i}H(z,w))(F(w,z)+\textbf{i}H(w,z))dA(w)dA(z)\},
\end{eqnarray*}
where $F(z,w)=\sum_{i,j=1}^{3g-3}{a_{ij}\mu_{i}(w)\cdot \overline{\mu_{j}(z)}}$ and 
$H(z,w)=\sum_{i,j=1}^{3g-3}{b_{ij}\mu_{i}(w)\cdot \overline{\mu_{j}(z)}}$.
\end{small}
\end{proposition}

\begin{proof} Since $Q(A,B)=Q(B,A)$,
\begin{eqnarray*}
Q(A+B,A+B)=Q(A,A)+2Q(A,B)+Q(B,B).
\end{eqnarray*}
By proposition \ref{coor}, proposition \ref{cooxy}, and proposition \ref{cofon12} we have
\begin{eqnarray*}
&& Q(A+B,A+B)\\
&=&(\int_{X}D(F(z,z)-\overline{F(z,z)})(F(z,z)-\overline{F(z,z)})dA(z)\\
             &-&\int_{X}{D(H(z,z)+\overline{H(z,z)})(H(z,z)+\overline{H(z,z)})}dA(z)\\
             &+&2\textbf{i} \cdot \int_{X}{D(F(z,z)-\overline{F(z,z)})(H(z,z)+\overline{H(z,z)})}dA(z))\\
            (&-&2\cdot \int_{X\times X}G(z,w)|F(z,w))|^2dA(w)dA(z)\\
             &-&2\cdot \int_{X\times X}G(z,w)|H(z,w))|^2dA(w)dA(z)\\            
             &-&4\cdot \Im\{\int_{X\times X}G(z,w)F(z,w)\overline{H(z,w)})dA(w)dA(z)\})\\
             (&+&2\cdot \Re\{\int_{X\times X}G(z,w)F(z,w)F(w,z)dA(w)dA(z)\}\\
             &-&2\cdot \Re\{\int_{X\times X}G(z,w)H(z,w)H(w,z)dA(w)dA(z)\}\\
             &-&4\cdot \Im\{\int_{X\times X}G(z,w)F(z,w)H(w,z)dA(w)dA(z)\}).
\end{eqnarray*}
The sum of the first three terms is exactly
\begin{eqnarray*}
-4\int_{X}{D(\Im\{F(z,z)+\textbf{i}H(z,z)\})\cdot(\Im\{F(z,z)+\textbf{i}H(z,z)\})dA(z)}.
\end{eqnarray*}

Just as $|a+\textbf{i}b|^2=|a|^2+|b|^2+2\cdot \Im(a\cdot \overline{b})$, where $a$ and $b$ are two complex numbers, the sum of the second three terms is
exactly
\begin{eqnarray*}
-2\cdot \int_{X\times X}G(z,w)|F(z,w)+\textbf{i}H(z,w))|^2dA(w)dA(z).
\end{eqnarray*}  

For the last three terms, since 
\begin{eqnarray*}
\Im(F(z,w)\cdot H(w,z))=-\Re(F(z,w)\cdot(\textbf{i}H(w,z))),
\end{eqnarray*} 
the sum is exactly 

\begin{small}
\begin{eqnarray*}
2\cdot \Re\{\int_{X\times X}G(z,w)(F(z,w)+\textbf{i}H(z,w))(F(w,z)+\textbf{i}H(w,z))dA(w)dA(z)\}.
\end{eqnarray*}
\end{small}

\end{proof}

Furthermore, we have
\begin{theorem} \label{cooxymn}
Under the same conditions as in proposition \ref{cooxym}, $Q$ is non-positive definite on $Span\{\wedge^{2}T_{X}^1\Teich(S),\wedge^{2}T_{X}^2\Teich(S)\}$, and the zero level subsets of $Q(\cdot,\cdot)$ are $\{\sum_{ij}b_{ij}\frac{\partial}{\partial x_{i}}\wedge \frac{\partial}{\partial y_{j}};\ b_{ij}=-b_{ji}\}$ in $Span\{\wedge^{2}T_{X}^{1}\Teich(S),\wedge^{2}T_{X}^{2}\Teich(S)\}$.
\end{theorem} 

\begin{proof} Let us estimate the terms in proposition \ref{cooxym} separately. For the first term,
since $D$ is a positive operator, 
\begin{eqnarray*}
-\int_{X}{D(\Im\{F(z,z)+\textbf{i}H(z,z)\})\cdot(\Im\{F(z,z)+\textbf{i}H(z,z)\})dA(z)}\leq 0.
\end{eqnarray*}

For the third term, by the Cauchy-Schwarz inequality,
\begin{small}
\begin{eqnarray*}
&&|\int_{X\times X}G(z,w)(F(z,w)+\textbf{i}H(z,w))(F(w,z)+\textbf{i}H(w,z))dA(w)dA(z)|\\
&\leq& \int_{X\times X}G(z,w)|(F(z,w)+\textbf{i}H(z,w))(F(w,z)+\textbf{i}H(w,z))|dA(w)dA(z)\\
&\leq& \sqrt{\int_{X\times X}G(z,w)|(F(z,w)+\textbf{i}H(z,w))|^2dA(w)dA(z)} \\
&\times& \sqrt{\int_{X\times X}G(z,w)|(F(w,z)+\textbf{i}H(w,z))|^2dA(w)dA(z)}\\
&=& \int_{X\times X}G(z,w)|(F(z,w)+\textbf{i}H(z,w))|^2dA(w)dA(z).
\end{eqnarray*}
\end{small}
The last equality follows from $G(z,w)=G(w,z)$.

Combining the two inequalities above and the second term in proposition \ref{cooxym}, we see that on $Span\{\wedge^{2}T_{X}^{1}\Teich(S),\wedge^{2}T_{X}^{2}\Teich(S)\}$ $Q$ is non-positive definite. Furthermore, $Q(A+B,A+B)=0$ if and only if there exists a constant $k$ such that both of the following hold:
\begin{eqnarray*}
\begin{cases}
Im\{F(z,z)+\textbf{i}H(z,z)\}=0, \\
F(z,w)+\textbf{i}H(z,w)=k\cdot\overline{(F(w,z)+\textbf{i}H(w,z))}.
\end{cases}
\end{eqnarray*}

If we let $z=w$, we get $k=1$. Hence, the second equation above is 
\begin{eqnarray*}
\sum_{ij}(a_{ij}-a_{ji}+\textbf{i}(b_{ij}+b_{ji}))\mu_{i}(w)\overline{\mu_{j}}(z)=0.
\end{eqnarray*} 
Since $\{\mu_{i}\}_{i\geq 1}$ is a basis,
\begin{eqnarray*}
a_{ij}=a_{ji},\quad \quad b_{ij}=-b_{ji}.
\end{eqnarray*}
That is, $A=0$ and $B=\sum_{ij}b_{ij}\frac{\partial}{\partial x_{i}}\wedge \frac{\partial}{\partial y_{j}}$, where $ b_{ij}=-b_{ji}$.
 
Conversely, if $A=0$ and $B\in \{\sum_{ij}b_{ij}\frac{\partial}{\partial x_{i}}\wedge \frac{\partial}{\partial y_{j}};\ b_{ij}=-b_{ji}\}$, it is not hard to apply proposition \ref{cooxym} to show that $Q(A+B,A+B)=0$.
\end{proof}

Before we prove the main theorem, let us define a natural action of $\textbf{J}$ on $\wedge^{2}T_{X}\Teich(S)$ by
\begin{eqnarray*}
\begin{cases}
\textbf{J}\circ\frac{\partial}{\partial x_{i}}\wedge \frac{\partial}{\partial x_{j}}:=\frac{\partial}{\partial y_{i}}\wedge \frac{\partial}{\partial y_{j}},\\
\textbf{J}\circ\frac{\partial}{\partial x_{i}}\wedge \frac{\partial}{\partial y_{j}}:=-\frac{\partial}{\partial y_{i}}\wedge \frac{\partial}{\partial x_{j}}=\frac{\partial}{\partial x_{j}}\wedge \frac{\partial}{\partial y_{i}},\\
\textbf{J}\circ\frac{\partial}{\partial y_{i}}\wedge \frac{\partial}{\partial y_{j}}:=\frac{\partial}{\partial x_{i}}\wedge \frac{\partial}{\partial x_{j}},
\end{cases}
\end{eqnarray*}
and extend it linearly. It is easy to see that $\textbf{J}\circ\textbf{J}=id.$

Now we are ready to prove theorem \ref{conp}.

\begin{proof}[Proof of Theorem \ref{conp}]
It follows from proposition \ref{corf} and theorem \ref{cooxymn} that $Q$ is non-positive definite. 

If $A=C-\textbf{J}\circ C$ for some a $C \in \wedge^{2}T_{X}\Teich(S)$. Then it is easy to see that $Q(A,A)=0$ since $\textbf{J}$ is an isometry,. 

Assume that $A \in \wedge^{2}T_{X}\Teich(S)$ such that $Q(A,A)=0$. Since $\wedge^{2}T\Teich(S)=Span\{\frac{\partial}{\partial x_{i}}\wedge \frac{\partial}{\partial x_{j}}, 
\frac{\partial}{\partial x_{k}}\wedge \frac{\partial}{\partial y_{l}}, \frac{\partial}{\partial y_{m}}\wedge \frac{\partial}{\partial y_{n}}\}$, there exists $a_{ij},b_{ij}$, and $c_{ij}$ such that 
\begin{eqnarray*}
A=\sum_{ij}a_{ij}\frac{\partial}{\partial x_{i}}\wedge \frac{\partial}{\partial x_{j}}+b_{ij}\frac{\partial}{\partial x_{i}}\wedge \frac{\partial}{\partial y_{j}}+ c_{ij}\frac{\partial}{\partial y_{i}}\wedge \frac{\partial}{\partial y_{j}}.
\end{eqnarray*}
Since $Q(A,A)=0$, by proposition \ref{corf} and theorem \ref{cooxymn} we must have
\begin{eqnarray*}
a_{ij}+c_{ij}=a_{ji}+c_{ji}, b_{ij}=-b_{ji}.
\end{eqnarray*}
That is,
\begin{eqnarray*}
a_{ij}-a_{ji}=-(c_{ij}-c_{ji}), b_{ij}=-b_{ji}.
\end{eqnarray*}
Set
\begin{eqnarray*}
C=\sum_{ij}a_{ij}\frac{\partial}{\partial x_{i}}\wedge \frac{\partial}{\partial x_{j}}+\frac{b_{ij}}{2}\frac{\partial}{\partial x_{i}}\wedge \frac{\partial}{\partial y_{j}}.
\end{eqnarray*}

\textbf{Claim}: $A=C-\textbf{J}\circ C$.\\

Since $\sum_{ij} a_{ij}\frac{\partial}{\partial x_{i}}\wedge \frac{\partial}{\partial x_{j}}=\sum_{i<j}(a_{ij}-a_{ji})\frac{\partial}{\partial x_{i}}\wedge \frac{\partial}{\partial x_{j}}$, we have
\begin{eqnarray*}
&&\textbf{J}\circ \sum_{ij}a_{ij}\frac{\partial}{\partial x_{i}}\wedge \frac{\partial}{\partial x_{j}}=\sum_{i<j}(a_{ij}-a_{ji})\frac{\partial}{\partial y_{i}}\wedge \frac{\partial}{\partial y_{j}}\\
&=&-\sum_{i<j}(c_{ij}-c_{ji})\frac{\partial}{\partial y_{i}}\wedge \frac{\partial}{\partial y_{j}}=-\sum{c_{ij}\frac{\partial}{\partial y_{i}}\wedge \frac{\partial}{\partial y_{j}}}.
\end{eqnarray*}
Similarly, 
\begin{eqnarray*}
\textbf{J}\circ\sum(\frac{b_{ij}}{2}\frac{\partial}{\partial x_{i}}\wedge \frac{\partial}{\partial y_{j}})=-\sum\frac{b_{ij}}{2}\frac{\partial}{\partial x_{i}}\wedge \frac{\partial}{\partial y_{j}}.
\end{eqnarray*}
The claim follows from the two equations above.
\end{proof}

\section{Harmonic maps into $\Teich(S)$}\label{5}

In this section we study the twist-harmonic maps from some domains into the Teichm\"uller space. Before we go to the rank-one hyperbolic space case, let us state the following lemma, which is influenced by lemma $5$ in \cite{Yeung}. 
\begin{lemma}\label{51}
Let $M$ be either $H_{Q,m}=Sp(m,1)/Sp(m)\cdot Sp(1)$ or $H_{O,2}=F_{4}^{-20}/SO(9)$. Then the rank-one Hyperbolic space $M$ cannot be totally geodesically immersed into $\Teich(S)$.
\end{lemma}
\begin{proof}
For $H_{Q,m}=Sp(m,1)/Sp(m)$, we assume that there is a totally geodesic immersion of $H_{Q,m}$ into $\Teich(S)$. We may select $p\in H_{Q,m}$. Choose a quaternionic line $l_{Q}$ on $T_{p}H_{Q,m}$, and we may assume that $l_{Q}$ is spanned over $R$ by $v,Iv,Jv$, and $Kv$. Without loss of generality, we may assume that $J$ on $l_{Q}\subset T_{p}H_{Q,m}$ is the same as the complex structure on $\Teich(S)$. Choose an element 
\begin{eqnarray*}
v\wedge Jv+Kv\wedge Iv \in \wedge^2 T_{p}H_{Q,m}.
\end{eqnarray*}
Let $Q^{H_{Q,m}}$ be the curvature operator on $H_{Q,m}$. 
\begin{small}
\begin{eqnarray*}
& Q^{H_{Q,m}}(v\wedge Jv+Kv\wedge Iv,v\wedge Jv+Kv\wedge Iv)=\\
&R^{H_{Q,m}}(v,Jv,v,Jv)+R^{H_{Q,m}}(Kv,Iv,Kv,Iv)+2\cdot R^{H_{Q,m}}(v,Jv,Kv,Iv).
\end{eqnarray*}
\end{small}
Since $I$ is an isometry, we have 
\begin{eqnarray*}
R^{H_{Q,m}}(Kv,Iv,Kv,Iv)&=&R^{H_{Q,m}}(IKv,IIv,IKv,IIv)\\
&=&R^{H_{Q,m}}(-Jv,-v,-Jv,-v)\\
&=&R^{H_{Q,m}}(v,Jv,v,Jv). 
\end{eqnarray*}
Similarly,
\begin{eqnarray*}
R^{H_{Q,m}}(v,Jv,Kv,Iv)&=&R^{H_{Q,m}}(v,Jv,IKv,IIv)\\
&=&R^{H_{Q,m}}(v,Jv,-Jv,-v)\\
&=&-R^{H_{Q,m}}(v,Jv,v,Jv).
\end{eqnarray*}
Combining the terms above, we have
\begin{eqnarray*}
Q^{H_{Q,m}}(v\wedge Jv+Kv\wedge Iv,v\wedge Jv+Kv\wedge Iv)=0.
\end{eqnarray*}
Since $f$ is a geodesical immersion,
\begin{eqnarray*}
Q^{\Teich(S)}(v\wedge Jv+Kv\wedge Iv,v\wedge Jv+Kv\wedge Iv)=0.
\end{eqnarray*}
On the other hand, by theorem \ref{conp}, there exists $C$ such that
\begin{eqnarray*}
v\wedge Jv+Kv\wedge Iv=C-\textbf{J}\circ C.
\end{eqnarray*}
Hence,
\begin{eqnarray}\label{16}
&&\textbf{J}\circ(v\wedge Jv+Kv\wedge Iv)\\
\nonumber &=&\textbf{J}\circ(C-\textbf{J}\circ C)=\textbf{J}\circ C-\textbf{J}\circ\textbf{J}\circ C=\textbf{J}\circ C-C\\
\nonumber &=&-(v\wedge Jv+Kv\wedge Iv).
\end{eqnarray}
On the other hand, since \textbf{J} is the same as $J$ in $H_{Q,m}$, we also have
\begin{eqnarray}\label{17}
&&\textbf{J}\circ(v\wedge Jv+Kv\wedge Iv)=(Jv\wedge JJv+JKv\wedge JIv)\\
\nonumber &=&Jv\wedge(-v)+Iv\wedge(-Kv)=v\wedge Jv+Kv\wedge Iv.
\end{eqnarray}

From equations (\ref{16}) and (\ref{17}) we get
\begin{eqnarray*}
v\wedge Jv+Kv\wedge Iv=0,
\end{eqnarray*} 
which is a contradiction since $l_{Q}$ is spanned over $R$ by $v,Iv,Jv$, and $Kv$.

In the case of the Cayley hyperbolic plane $H_{O,2}=F_{4}^{20}/SO(9)$, the argument is similar by 
replacing a quaternionic line by a Cayley line (\cite{Chavel}). 
\end{proof}
 
Now we are ready to prove theorem \ref{hmrt}.

\begin{proof}[Proof of theorem \ref{hmrt}]
Since the sectional curvature operator on $\Teich(S)$ is non-positive definite, $\Teich(S)$ also has non-positive Riemannian sectional curvature in the complexified sense as stated in \cite{MSY}. Suppose that $f$ is not constant. From theorem 2 in \cite{MSY} (also see \cite{Cor}), we know that $f$ should be a totally geodesic immersion, which contradicts lemma \ref{51}. Hence, $f$ must be a constant. 
\end{proof}

\begin{remark}
In \cite{Yeung} it is shown that the image of any homomorphism $\rho$ from $\Gamma$ to $\Mod(S)$ is finite. Hence, $\rho(\Gamma)$ must have a fixed point in $\Teich(S)$ from the Nielsen realization theorem (one can see \cite{SKerck83,Wolpert5}). If we assume that there exists a twist harmonic map $f$ with respect to this homomorphism, then by theorem \ref{hmrt} we know $\rho(\Gamma) \subset \Mod(S)$ will fix the point $f(G/\Gamma) \in \Teich(S)$. 
\end{remark}

\begin{remark}
Conversly, if one can prove that for any homomorphism $\rho$ from $\Gamma$ to $\Mod(S)$ there exists a twist harmonic map $f$ from $G$ into the completion $\overline{\Teich(S)}$ of $\Teich(S)$ such that the image $f(G)\subset \Teich(S)$, theorem \ref{hmrt} tells us that the image $\rho(\Gamma)$ fixes a point in $\Teich(S)$; hence, the image $\rho(\Gamma)$ is finite because $\Mod(S)$ acts properly on $\Teich(S)$.   
\end{remark}

\end{document}